\documentclass[12pt]{amsart}
\usepackage{amssymb}
\usepackage{longtable}

\newtheorem{thm}{Theorem}[section]
\newtheorem{lem}[thm]{Lemma}

\newtheorem{ex}[thm]{Example}

\newtheorem*{prob*}{Open problem}

\theoremstyle{definition}

\theoremstyle{remark}

\newtheorem*{rem*}{Remark}

\newcommand{\R}{\mathbb{R}}
\newcommand{\C}{\mathbb{C}}

\newcommand{\g}{\mathfrak{g}}
\newcommand{\hh}{\mathfrak{h}}
\newcommand{\Aut}{\mathrm{\mathop{Aut}}}

\newcommand{\wt}[1]{\widetilde{#1}}
\newcommand{\ad}{\mathrm{\mathop{ad}}}

\begin{document}


\title{Classification of Novikov algebras}

\author[D. Burde]{Dietrich Burde}
\thanks{The first author was supported by the FWF, Projekt P21683.}
\author[W. de Graaf]{Willem de Graaf}
\thanks{The second author thanks the ESI for its hospitality and support.}
\address{Fakult\"at f\"ur Mathematik\\
Universit\"at Wien\\
  Nordbergstr. 15\\
  1090 Wien \\
  Austria}
\email{dietrich.burde@univie.ac.at}
\address{Dipartimento di Matematica\\
Via Sommarive 14\\
I-38050 Povo (Trento)\\
Italy}
\email{degraaf@science.unitn.it}

\date{\today}


\begin{abstract}
We describe a method for classifying the Novikov algebras with a given associated 
Lie algebra. Subsequently we give the classification of the Novikov algebras of
dimension 3 over $\R$ and $\C$, as well as the classification of the 4-dimensional
Novikov algebras over $\C$ whose associated Lie algebra is nilpotent. In particular
this includes a list of all 4-dimensional commutative associative algebras over $\C$.
\end{abstract}

\maketitle

\section{Introduction}

Novikov algebras arise in many areas of mathematics and physics.
They form a special class of pre-Lie algebras which arise among
other things in the study of rooted trees (Cayley), convex homogeneous
cones (Vinberg), affinely flat manifolds and their fundamental groups
(Auslander, Milnor) and renormalization theory (Connes, Kreimer, Kontsevich). 
Novikov algebras in particular were considered in the study of
Hamiltonian operators, Poisson brackets of hydrodynamic type (Balinskii, Novikov),
operator Yang-Baxter equations and left-invariant affine structures on Lie
groups \cite{BU22}. We refer to \cite{BU1}, and the references therein, 
for more information on these topics.

The theory of Novikov algebras and its classification has been started by Zelmanov
\cite{ZEL}. Further structure theory has been developed in \cite{BU32}.
There have been several efforts to classify complex Novikov algebras in low 
dimensions: Bai and Meng classify the Novikov algebras over $\C$ of dimension up to 3
in \cite{BAI1}, and complete Novikov algebras over $\C$ that have a nilpotent
associated Lie algebra in \cite{BAI2}. We recall that a Novikov algebra is said to be 
complete, or transitive, if the right multiplication by any element is a nilpotent linear map.
Also we note that Novikov algebras are Lie admissible algebras, i.e., the commutator defines the 
structure of a Lie algebra on them. In this paper we outline a systematic method to
classify the Novikov algebras with a given associated Lie algebra. Using it we obtain
classifications that extend the ones by Bai and Meng in several ways. Firstly, we get
the classification of the 3-dimensional Novikov algebras over $\R$ and $\C$ (Section
\ref{sec:dim3}). Secondly (Section \ref{sec:dim4}) we
have classified all 4-dimensional Novikov algebras over $\C$ that have a nilpotent associated
Lie algebra (not just the complete ones). Furthermore, using our method we obtain
a list of Novikov algebras that is ordered with respect to the associated Lie algebra,
i.e., all Novikov algebras with the same Lie algebra are grouped together.

If the associated Lie algebra of a Novikov algebra is abelian, then
the Novikov algebra is nothing other than a commutative and associative algebra. 
This is an interesting class of algebras in its own. It also appears as
a subclass of other classes of algebras being of interest in geometry; as an example
we name LR-algebras \cite{BU34}. Hence an explicit classification
of associative, commutative algebras in low dimension is also very desirable.
Over the past century many classifications of 
commutative associative algebras have appeared in the
literature. Here we mention \cite{BDM}, \cite{gra_ass}, \cite{nilrings}, \cite{poonen}.
The last three of these references have classifications of nilpotent associative algebras.
In Section \ref{sec:caa} we describe an elementary way to obtain the classification
of all commutative associative algebras from the classification of the nilpotent
commutative associative algebras. We apply this method to obtain a list of 
commutative associative algebras of dimensions 3 (over $\R$ and $\C$) and 4 (over $\C$; 
the latter is contained in Section \ref{sec:dim4}). The list of commutative associative
algebras of dimension 3 over $\C$ is also contained in \cite{BDM}. For completeness we
also include it in this paper. Although the classification of the commutative
associative algebras can be deduced from the classification of the nilpotent
commutative associative algebras, no such classification in dimension 4 has explicitly
been done, to the best of our knowledge. 

We start, in the next section, by describing the method that we use for 
classifying Novikov algebras. In order to compute isomorphisms between these
algebras, we rely on computer calculations, using the technique of Gr\"obner bases.
Using this technique we have also established the correspondence between our list 
of 3-dimensional Novikov algebras over $\C$ and the list in \cite{BAI1}; therefore, the
two classifications are equivalent.

Finally we remark that it turned out that the problem of classifying all 4-dimensional
Novikov algebras is far too complex to be undertaken. Therefore, we have restricted ourselves
to the Novikov algebras with a nilpotent associated Lie algebra. We remark that many of
the Novikov algebras that we get are not complete (for example $N^{\hh_1}_{17}$ in Table
\ref{tab3}). Therefore our classification substantially extends the one in \cite{BAI2}.

\section{Classifying Novikov algebras with given Lie algebra}\label{sec:1}

Let $F$ be a field.
An algebra $A$ over $F$ is called Novikov if
\begin{align}
x\cdot (y\cdot z) - (x\cdot y)\cdot z &= y\cdot (x\cdot z) - (y\cdot x)
\cdot z \label{nov1}\\
(x\cdot y)\cdot z &= (x\cdot z)\cdot y,\label{nov2}
\end{align}
for all $x,y,z\in A$. Let $A$ be a Novikov algebra, and 
define a bracket on $A$ by 
$$[x,y] = x\cdot y -y\cdot x.$$
Then it is straightforward to see that $(A,[~,~])$ is a Lie algebra.
The problem considered here is to find, up-to isomorphism, 
all Novikov algebras that have a
given Lie algebra as associated Lie algebra.

Let $\g$ be a finite-dimensional Lie algebra. Let $\theta : \g\times \g\to 
\g$ be a bilinear map. Then we define an algebra $A(\g,\theta)$ as follows.
The underlying vector space of $A(\g,\theta)$ is $\g$. Furthermore, for
$x,y\in \g$ we set $x\cdot y = \theta(x,y)$. We define $T(\g)$ to be the set
of all bilinear $\theta$ such that $[x,y] = \theta(x,y)-\theta(y,x)$ for all
$x,y\in \g$ and such that $A(\g,\theta)$ is a Novikov algebra.

The automorphism group, $\Aut(\g)$, of $\g$ acts on $T(\g)$ by 
$\phi\theta(x,y) = \phi^{-1}( \theta(\phi(x),\phi(y)))$, for $\phi\in \Aut(\g)$,
and $\theta\in T(\g)$.

\begin{lem}\label{lem:1}
Let $\theta_1,\theta_2\in T(\g)$. Then $A(\g,\theta_1)$ and $A(\g,\theta_2)$
are isomorphic if and only if there is a $\phi\in \Aut(\g)$ with $\phi\theta_1
= \theta_2$. 
\end{lem}

\begin{proof}
If $\phi \theta_1 = \theta_2$, then $\phi (\theta_2(x,y) ) = \theta_1( \phi(x),
\phi(y))$. Hence $\phi : A(\g,\theta_2) \to A(\g,\theta_1)$ is an isomorphism.

On the other hand, if $\phi : A(\g,\theta_2) \to A(\g,\theta_1)$ is an 
isomorphism of Novikov algebras, then $\phi \theta_2(x,y) = \theta_1(\phi(x),
\phi(y))$ for all $x,y\in \g$. Hence $\phi\theta_1 = \theta_2$.
This also implies that $\phi$ is an automorphism
of $\g$.
\end{proof}

Set $n=\dim \g$, and let $e_1,
\ldots, e_n$ be a basis of $\g$. Then an element
in $T(\g)$ is given by a sequence of $n\times n$-matrices $L_1,\ldots,L_n$, 
where $L_i$ describes the left multiplication of $e_i$. More precisely,
if $\theta\in T(\g)$ satisfies $\theta(e_i,e_j) = \sum_{k=1}^n c_{ij}^k
e_k$, then for the $L_i$ we have $L_i(k,j) = c_{ij}^k$. 

This way we view $T(\g)$ is an affine variety in
$F^{n^3}$. We get equations for this variety by plugging $x=e_i$, $y=e_j$,
$z=e_k$ in (\ref{nov1}) and (\ref{nov2}), for $1 \leq i,j,k\leq n$, and by 
requiring that $e_i\cdot e_j -e_j\cdot e_i = [e_i,e_j]$ for $i<j$. The
equations corresponding to this last requirement will be linear. The
equations corresponding to (\ref{nov1}) and (\ref{nov2}) will be polynomial;
however we can reformulate (\ref{nov2}) such that it leads to linear equations
as well. This works as follows. Let $A$ be an algebra. 
For $x\in A$ let $L(x) : A\to A$ be the
linear map given by $L(x)(y) = x\cdot y$. Similarly we define $R(x)(y) =
y\cdot x$. The adjoint map is defined by $\ad(x)= L(x)-R(x)$.
Now assume that $A$ satisfies (\ref{nov1}). Then (\ref{nov2}) holds
as well if and only if $[R(x),R(y)]=0$ for all $x,y\in A$. Furthermore,
(\ref{nov1}) implies that $[L(x),L(y)] = L([x,y])$. Now, using $R(x) = 
L(x)-\ad(x)$ we get that $[R(x),R(y)] = 0$ is equivalent to
\begin{equation}\label{nov3}
L([x,y])+\ad([x,y])-[L(x),\ad(y)]-[\ad(x),L(y)]=0.
\end{equation}
If in this equation we set $x=e_i$, $y=e_j$, and noting that the maps
$\ad e_k$ are given by the Lie multiplication on $\g$, we get that 
(\ref{nov3}) is equivalent to a set of linear equations.

\begin{ex}\label{exa:1}
Consider the 2-dimensional Lie algebra $\g$ over $F=\C$ 
with basis $e_1,e_2$ and 
Lie bracket $[e_1,e_2]=e_1$. Write $L_i = L(e_i)$ and 
$$ L_1 = \begin{pmatrix} a_{11} & a_{12} \\ a_{21} & a_{22} \end{pmatrix},~~
L_2 = \begin{pmatrix} b_{11} & b_{12} \\ b_{21} & b_{22} \end{pmatrix}.$$
Then (\ref{nov3}) with $x=e_1$, $y=e_2$ is equivalent to the equations
$a_{21}=0$, $a_{11}=b_{21}$, $a_{22}=-b_{21}$, $b_{22}= b_{11}+1$.
Furthermore, $L(e_1)(e_2)-L(e_2)(e_1)=[e_1,e_2]$ is tantamount to 
$a_{22}=b_{21}$, $a_{12}=b_{11}+1$. So the linear equations that we get imply
that
\begin{equation}\label{exa1:Li}
L_1 = \begin{pmatrix} 0 & b_{22} \\ 0 & 0 \end{pmatrix},~~
L_2 = \begin{pmatrix} b_{22}-1 & b_{12} \\ 0 & b_{22} \end{pmatrix}.
\end{equation}
These matrices already define a Novikov structure, i.e., (\ref{nov1})
is automatically satisfied.
\end{ex}

By Lemma \ref{lem:1}, classifying the Novikov algebras with associated
Lie algebra equal to $\g$ is the same as listing the $\Aut(\g)$-orbits
on $T(\g)$. In order to carry this out we write an element of $\Aut(\g)$ 
as an $n\times n$-matrix, with indeterminates as entries that satisfy some
polynomial equations. Then we try and work out what the orbits are. We
illustrate this by example.

\begin{ex}\label{exa:2}
We consider the situation of Example \ref{exa:1}. An element of $\Aut(\g)$
is given by 
$$\phi = \begin{pmatrix} a & b \\ 0 & 1 \end{pmatrix},$$
where $a\neq 0$. (Here we use the column convention, so $\phi(e_1) =ae_1$,
$\phi(e_2) = be_1+e_2$.) Let a $\theta\in T(\g)$ be given by the matrices
(\ref{exa1:Li}). Then a short calculation shows that $\phi\theta$ corresponds
to two matrices of the same shape, where $b_{22}$ is unchanged, but where
$b_{12}$ is changed into 
$$a^{-1}(b(b_{22}-1)+b_{12}).$$
We now distinguish two cases. In the first case $b_{22}\neq 1$. Then we 
can choose $b$ so that $b_{12}$ is mapped to 0. So we get a $1$-parameter
family of Novikov algebras given by 
$$ e_1\cdot e_2 = b_{22}e_1,~ e_2\cdot e_1 = (b_{22}-1)e_1, ~ 
e_2\cdot e_2 = b_{22}e_2.$$
Algebras corresponding to different values of the parameter $b_{22}$ are
not isomorphic. 

In the second case we have $b_{22}=1$. Here we have two subcases. If $b_{12}=0$
then we get an algebra which is included in the above parametric family.
On the other hand, if $b_{12}\neq 0$ then we can choose $a=b_{12}$ and
we see that it is changed to $1$. So we get one more algebra, given by
$$ e_1\cdot e_2 = e_1,~ e_2\cdot e_2 = e_1+e_2.$$
\end{ex}

Summarising, our procedure to classify the Novikov algebras with
given associated Lie algebra $\g$ consists of two steps:

\begin{enumerate}
\item Find the equations for $T(\g)$.
\item List the orbits of $\Aut(\g)$ on $T(\g)$. 
\end{enumerate}

The last step is by far the most difficult one. And we are not always
able to carry it out in full; in other words, sometimes we obtain two
Novikov algebras that are isomorphic without us being able to show that
the corresponding bilinear maps in $T(\g)$ lie in the same $\Aut(\g)$-orbit.
To deal with a situation of this kind we use a method based on the algorithmic 
technique of Gr\"obner bases (cf. \cite{clo}). For this again we write an
element of $\Aut(\g)$ as a matrix with entries that are indeterminates,
satisfying some polynomial equations. Let $R$ denote the ring containing
these indeterminates. Let $\theta_1,\theta_2\in T(\g)$ be
given. Then we compute the matrices $L_i$ corresponding to $\phi\theta_1$;
they have entries that are polynomials in the coefficients of $\phi$
(which are indeterminates). Then the requirement that $\phi\theta_1 = \theta_2$
leads to a set of polynomial equations $p=0$ where $p\in P\subset R$. 
Let $I$ be the ideal of $R$ generated by $P$. Now, solving $p=0$ for $p\in P$
is the same as solving $g=0$ for any generating set $G$ of $I$. For this
a Gr\"obner basis is particularly convenient. In particular if the Gr\"obner
basis is computed relative to a lexicographical order, then the resulting
polynomial equations have a triangular structure, which makes them easier
to solve. A second feature of this method is that if the algebras corresponding
to the $\theta_i$ happen {\em not} to be isomorphic, then the reduced Gr\"obner
basis is $\{1\}$. So in this case no hand calculations are necessary.
For the computation of the Gr\"obner bases we used the computer algebra system
{\sc Magma} (\cite{magma}). We illustrate this with an example.

\begin{ex}
Let $\g$ be the 3-dimensional Lie algebra with basis $e_1,e_2,e_3$ and
nonzero bracket $[e_1,e_2]=e_3$. Consider a family of Novikov algebras
$N_1^\alpha$ given by $e_1\cdot e_2 = (\alpha+1)e_3$, $e_2\cdot e_1 = \alpha e_3$. 
Denote the corresponding element of $T(\g)$ by $\theta_\alpha$. The elements
of $\Aut(\g)$ are given by 
$$ \phi = \begin{pmatrix}x_{11} & x_{12} & 0 \\ x_{21} & x_{22} & 0 \\
x_{31} & x_{32} & \delta \end{pmatrix},$$
with $\delta = x_{11}x_{22}-x_{12}x_{21}$. The Novikov algebra corresponding
to $\phi\theta_\alpha$ has nonzero products
\begin{align*}
e_1\cdot e_1 &= (2D x_{11}x_{21}\alpha + D x_{11}x_{21})e_3\\
e_1\cdot e_2 &= (2D x_{12} x_{21}\alpha + D x_{12}x_{21} + \alpha + 1)e_3 \\
e_2\cdot e_1 &= (2Dx_{12}x_{21}\alpha + Dx_{12}x_{21} + \alpha)e_3\\
e_2\cdot e_2 &= (2Dx_{12}x_{22}\alpha + Dx_{12}x_{22})e_3,
\end{align*}
where $D = \delta^{-1}$. 

Now let $N_2^\beta$ be the family of Novikov algebras given by 
$e_1\cdot e_1=\beta e_3$, $e_1\cdot e_2 = e_3$, $e_2\cdot e_2=e_3$.
Then $N_1^\alpha \cong N_2^\beta$ if and only if the
following polynomial equations are satisfied
\begin{align*}
& 2D x_{11}x_{21}\alpha + D x_{11}x_{21}-\beta=0\\
& 2D x_{12} x_{21}\alpha + D x_{12}x_{21} + \alpha =0\\
& 2Dx_{12}x_{22}\alpha + Dx_{12}x_{22}-1 =0.
\end{align*}
We let $I$ be the ideal of the polynomial ring $\C[D,x_{11},x_{12},x_{21},x_{22},
\alpha,\beta]$ generated by the left hand sides of these equations along
with the polynomial $D(x_{11}x_{22}-x_{12}x_{21})-1$. A reduced Gr\"obner basis
of $I$ with respect to the lexicographical order with $ D > x_{11} > x_{12} >
\cdots > \alpha >\beta$ consists of the polynomials
\begin{align*}
&    D x_{12} x_{22} \alpha + \tfrac{1}{2} D x_{12} x_{22} - \tfrac{1}{2},\\
&    D x_{12} x_{22} \beta - \tfrac{1}{4} D x_{12} x_{22} + \tfrac{1}{2} 
\alpha + \tfrac{1}{4},\\
&    x_{11} - x_{12} \alpha - x_{12},\\
&    x_{21} + x_{22} \alpha,\\
&    \alpha^2 + \alpha + \beta.
\end{align*}
It follows that $N_1^\alpha\cong N_2^\beta$ if and only if these polynomial equations have a
solution. So
from the last polynomial we see that $N_1^\alpha\cong N_2^\beta$ implies 
that $\beta = -\alpha^2-\alpha$. Conversely, suppose that this is satisfied. From the 
third and fourth elements of the Gr\"obner basis we get that an isomorphism
$\phi$ has to be of the form
$$ \begin{pmatrix} (\alpha+1)u & u & 0 \\ -\alpha v & v & 0 \\
x_{31} & x_{32} & \delta \end{pmatrix}.$$
The determinant of the $2\times 2$-block in the upper left corner is
$(2\alpha +1)uv$. Now if $\alpha \neq -\tfrac{1}{2}$ then we choose $u=v=1$, and $x_{31}=x_{32}=0$
and get the linear map
$\phi : N_2^{-\alpha^2-\alpha} \to N_1^{\alpha}$ given by 
$$ \begin{pmatrix} \alpha+1 & 1 & 0 \\ -\alpha  & 1 & 0 \\
0 & 0 & 2\alpha +1 \end{pmatrix}.$$
It is straightforward to check that in fact this is an isomorphism.

There remains the case where $\alpha = -\tfrac{1}{2}$. Adding the 
polynomial $\alpha +\tfrac{1}{2}$ to the generating set of the ideal
we get that the Gr\"obner basis is $\{1\}$. Hence in this case the algebras
are not isomorphic. 
\end{ex}

\section{Novikov algebras with an abelian Lie algebra}\label{sec:caa}

If the associated Lie algebra of a Novikov algebra is abelian, then
the Novikov algebra is a commutative associative algebra (CAA).
Conversely, every CAA is a Novikov algebra with abelian Lie algebra.
In this section we describe how to obtain a classification of the CAA's
of dimension $3$ over $\R$ and $\C$. 

First we introduce some notation, and recall some facts on associative algebras.
We refer to \cite{drozd_kir} for an in-depth account of these matters.

Let $A$ be an associative algebra over a field $F$ of characteristic 0.
If $A$ does not have a $1$, then we write $\wt{A}$ for the
algebra $A\oplus \langle 1 \rangle$ (with $a\cdot 1 = 1\cdot a = a$ for all $a\in A$).
If there are nontrivial proper ideals $B,C$ of $A$ such that 
$A=B\oplus C$ (direct sum of vector spaces), then $A$ is said to be a direct sum. 
(Note that necessarily $BC=CB=0$.)
The algebra $A$ is said to be nilpotent if there is an
$m>0$ such that $a_1\cdots a_m= 0$ for all $a_1,\ldots,a_m\in A$.
The radical $R$ of $A$ is its largest nilpotent ideal; and $A$ is said to be semisimple
if $R=0$. More generally, if $A$ has a one, then there exists a semisimple subalgebra
$S$ of $A$ such that $A=S\oplus R$ (direct sum of vector spaces), and $S\cong A/R$.
Furthermore, if $S$ is a semisimple commutative algebra, then it is a direct sum: 
$S= K_1\oplus \cdots \oplus K_m$, where each $K_i$ is a field extension of $F$.
This decomposition corresponds to a decomposition of the identity element: 
$1=\epsilon_1+\cdots +\epsilon_m$, where the $\epsilon_i$ are orthogonal primitive
idempotents and $K_i = \epsilon_i S$.
In particular, if the base field is $\R$, then $K_i$ is either isomorphic to $\R$ or
to $\C$. We write $\C$ for the commutative associative algebra over $\R$ with basis $e_1,e_2$ and
non-zero products $e_1^2=e_1$, $e_1e_2=e_2e_1=e_2$, $e_2^2=-e_1$.

\begin{lem}\label{lem:caa}
Let $A$ be a commutative associative algebra over $\R$. Suppose that $A$ is not a 
direct sum of proper non-trivial ideals. Then $A$ is either
\begin{itemize}
\item nilpotent, or
\item equal to $\wt{B}$ for a nilpotent $B$, or
\item equal to $\C\oplus R$, where $R$ is the radical.
\end{itemize}
\end{lem}

\begin{proof}
If $A$ has a $1$ then $A = S\oplus R$, where $S$ is semisimple, and
$R$ is the radical. As above $S=K_1\oplus \cdots \oplus K_m$, where
the $K_i$ are field extensions of $\R$. Let $1=\epsilon_1+\cdots +\epsilon_m$
be the corresponding decomposition of $1$, where $K_i = \epsilon_i S$.
Then $A$ is the direct sum of the ideals $\epsilon_i A$. If $m>1$, they are all nontrivial
and proper. Hence it follows that $m=1$ and $S$ is a field. If $S\cong \R$ then
$A\cong \wt{R}$. If $S\cong \C$ then we are in the third case.

If $A$ has no $1$, then we consider $\wt{A}$. Again we get a decomposition $\wt{A}
=S\oplus R$, and $m$ orthogonal idempotents $\epsilon_i$. Writing $\epsilon_i = \mu_i\cdot 1
+a_i$, with $\mu_i\in \R$ and $a_i\in A$ we see that the fact that the $\epsilon_i$ 
are orthogonal idempotents with sum $1$ implies
that $m-1$ of them (say $\epsilon_1,\ldots,\epsilon_{m-1}$) lie in $A$, and $\epsilon_m$ does
not. Again we get that $A$ is the direct sum of the ideals $\epsilon_iA$. If $\epsilon_i A=A$
for some $i\leq m-1$ then $\epsilon_i$ is an identity element in $A$. So this cannot happen.
Again it follows that $m=1$. If $S\cong \R$ then $A$ is nilpotent. If $S\cong \C$ then
a basis of $S$ is $e_1=1$, $e_2$ with $e_2^2=-e_1$. However, we can also write $e_2 =
\mu\cdot 1 + a$ for some $\mu\in \R$ and $a\in A$. From this we get that $\mu^2 =-1$.
We conclude that this case cannot occur. 
\end{proof}

From \cite{gra_ass} we get all real nilpotent CAA's of dimensions $\leq 3$. They are

\begin{longtable}{|r|l|l|}
\caption{Nilpotent CAA's of dimensions $\leq 3$ over $\R$.}\label{tab:nilcaas}
\endfirsthead
\hline
\multicolumn{3}{|l|}{\small\slshape Nilpotent CAA's.} \\ 
\hline
\endhead
\hline
\endfoot
\endlastfoot

\hline
$\dim$ & name & multiplication table \\

\hline
0 & $A_0$ & \\
\hline 
1 & $A_1$ &  \\
\hline 
2 & $A_{2,1}$ & \\
2 & $A_{2,2}$ & $e_1^2=e_2$\\
\hline

3 & $A_{3,1}$ & \\
3 & $A_{3,2}$ & $e_1^2=e_2$\\
3 & $A_{3,3}$ & $e_1^2=e_2$, $e_1e_2=e_3$\\
3 & $A_{3,4}$ & $e_1^2=e_3$, $e_2^2=e_3$ \\
3 & $A_{3,5}$ & $e_1^2=-e_3$, $e_2^2=e_3$ \\
\hline

\end{longtable}

Here, over $\C$ the algebras $A_{3,4}$ and $A_{3,5}$ are isomorphic.

Now, using Lemma \ref{lem:caa}, we get all CAA's of dimension $3$ over $\R$. They
are: $A_{3,i}$ for $1\leq i\leq 5$, $\wt{A_{2,i}}$, $i=1,2$, $\wt{A_0}\oplus \wt{A_0}\oplus
\wt{A_0}$, $\wt{A_0}\oplus \wt{A_0}\oplus A_1$, $\wt{A_0}\oplus A_1 \oplus A_1$,
$\wt{A_1}\oplus A_1$, $\wt{A_1}\oplus \wt{A_0}$, $\C\oplus A_1$, $\C\oplus \wt{A_0}$, 
$A_{2,2}\oplus \wt{A_0}$.

So we get 15 algebras in total. Over $\C$ we get 12 of them, as the pairs
$A_{3,4}$, $A_{3,5}$ and $\wt{A_0}\oplus \wt{A_0}\oplus
\wt{A_0}$, $\C\oplus \wt{A_0}$ and $\wt{A_0}\oplus \wt{A_0}\oplus A_1$, $\C\oplus A_1$
become isomorphic and there are no other isomorphisms.

\section{Novikov algebras of dimension three over $\R$ and $\C$}\label{sec:dim3}

A simple Lie algebra of dimension 3 does not have Novikov structures.
The Novikov algebras of dimension 3 with abelian Lie algebra were classified
in the previous section. So this leaves the classification of the Novikov
algebras of dimension 3, where the associated Lie algebra is solvable and
non-abelian.

From \cite{gra11} we get that over $\R$ and $\C$ there are the following
solvable Lie algebras:

\begin{center}
\begin{tabular}{|l|l|}
\hline 
name & nonzero brackets \\
\hline
$\g_1$ & $[e_1,e_2]=e_2$, $[e_1,e_3]=e_3$\\ 
$\g_2^\alpha$ & $[e_1,e_2] = e_3$, $[e_1,e_3]=\alpha e_2+e_3$ \\
$\g_3$ & $[e_1,e_2]=e_3$\\
$\g_4$ & $[e_1,e_2]= e_3$, $[e_1,e_3]=e_2$\\
$\g_5$ & $[e_1,e_2]= e_3$, $[e_1,e_3]=-e_2$\\
\hline
\end{tabular}
\end{center}

Among these algebras there are no isomorphisms, except that $\g_4$ and
$\g_5$ are isomorphic over $\C$, but not over $\R$. 

Next are the tables of Novikov algebras that we get. On some occasions we give a 
parametrised class of algebras. In those cases, if nothing is stated about isomorphisms,
then different values of the parameter give non-isomorphic Novikov algebras. Moreover,
the classification that we give is over $\R$. Over $\C$ some isomorphisms between elements
of the list arise, and those are explicitly given.
The Novikov algebras with associated Lie algebra $\g$ will be denoted $N_i^\g$, $i=1,2,\ldots$.
For the Lie algebra $\g_2^\alpha$ we have three tables with associated Novikov algebras.
In Table \ref{tab3_2} we list the Novikov algebras that we get for generic values
of the parameter $\alpha$. Table \ref{tab3_3} contains the extra Novikov algebras that
we get when $\alpha= -\tfrac{2}{9}$. And in Table \ref{tab3_4} we give the extra algebras
that arise when $\alpha=0$. 

\begin{longtable}{|l|l|}
\caption{Novikov algebras with Lie algebra $\g_1$.}\label{tab3_1}
\endfirsthead
\hline
\multicolumn{2}{|l|}{\small\slshape Novikov algebras corresponding to $\g_1$.} \\ 
\hline
\endhead
\hline
\endfoot
\endlastfoot

\hline
name & multiplication table \\

\hline

$N^{\g_1}_1(a)$ & $e_1e_1 = a e_1$, $e_1e_2 = (a+1)e_2$,
$e_1e_3 =(a+1)e_3$, $e_2e_1 = a e_2$, $e_3e_1=a e_3$ \\

$N^{\g_1}_2$ & $e_1e_1 = -e_1+e_2$, 
$e_2e_1 = - e_2$, $e_3e_1=- e_3$\\

\hline
\end{longtable}

\begin{longtable}{|l|l|}
\caption{Novikov algebras with Lie algebra $\g_2^\alpha$.}\label{tab3_2}
\endfirsthead
\hline
\multicolumn{2}{|l|}{\small\slshape Novikov algebras corresponding 
to $\g_2^\alpha$.} \\ 
\hline
\endhead
\hline
\endfoot
\endlastfoot

\hline
name & multiplication table \\

\hline

$N^{\g_2^\alpha}_1(a)$ & $e_1e_1 = a e_1$, $e_1e_2 = a e_2+e_3$,
$e_1e_3 = \alpha e_2+(a+1)e_3$, $e_2e_1 = a e_2$, $e_3e_1= a e_3$ \\

$N^{\g_2^{a^2+a}}_2(a)$ & $e_1e_1 = a e_1+e_2$, $e_1e_2 = a e_2+e_3$,
$e_1e_3 = (a^2+a) e_2+(a+1)e_3$, $e_2e_1 = a e_2$, \\
& $e_3e_1= a e_3$ \\

\hline
\end{longtable}

\begin{longtable}{|l|l|}
\caption{Novikov algebras with Lie algebra $\g_2^{-\tfrac{2}{9}}$.}\label{tab3_3}
\endfirsthead
\hline
\multicolumn{2}{|l|}{\small\slshape Novikov algebras corresponding 
to $\g_2^{-\tfrac{2}{9}}$.} \\ 
\hline
\endhead
\hline
\endfoot
\endlastfoot

\hline
name & multiplication table \\

\hline

$N^{\g_2^{-\tfrac{2}{9}}}_3$ & $e_1e_1 = -\tfrac{1}{3} e_1$, $e_1e_2 =  
\tfrac{8}{3} e_2-8 e_3$,
$e_1e_3 = \tfrac{7}{9} e_2 - \tfrac{7}{3}e_3$, $e_2e_1 = \tfrac{8}{3} e_2
-9e_3$, \\
& $e_3e_1= e_2-\tfrac{10}{3} e_3$ \\

$N^{\g_2^{-\tfrac{2}{9}}}_4$ & $e_1e_1 = -\tfrac{1}{3} e_1+e_2$, $e_1e_2 =  
\tfrac{8}{3} e_2-8 e_3$,
$e_1e_3 = \tfrac{7}{9} e_2 - \tfrac{7}{3}e_3$, $e_2e_1 = \tfrac{8}{3} e_2
-9e_3$, \\
& $e_3e_1= e_2-\tfrac{10}{3} e_3$ \\

$N^{\g_2^{-\tfrac{2}{9}}}_5(a)$ & $e_1e_1 = 3a e_1 + (-3a^2-\tfrac{1}{3}
a)e_2$, $e_1e_2 = 6a e_2 +(-9a +1)e_3$, \\
& $e_1e_3 = (a-\tfrac{2}{9} )e_2+e_3$,$e_2e_1 = 6a e_2 -9a e_3$, $e_2e_2 = -3e_2 +9e_3$, \\
& $e_2e_3 = -e_2+3e_3$, $e_3e_1 = a e_2$, 
$e_3e_2 = -e_2 +3e_3$, $e_3e_3 = -\tfrac{1}{3} e_2+e_3$\\ 

$N^{\g_2^{-\tfrac{2}{9}}}_6$ & $e_1e_1 = -\tfrac{2}{3} e_1 -\tfrac{8}{27}
e_2+\tfrac{2}{3}e_3$, $e_1e_2 = -\tfrac{4}{3} e_2 +3 e_3$, 
$e_1e_3 = -\tfrac{4}{9}e_2+e_3$,\\
& $e_2e_1 = -\tfrac{4}{3} e_2 +2 e_3$, $e_2e_2 = -3e_2 +9e_3$, 
$e_2e_3 = -e_2+3e_3$, $e_3e_1 = -\tfrac{2}{9}e_2$, \\
& $e_3e_2 = -e_2 +3e_3$, $e_3e_3 = -\tfrac{1}{3} e_2+e_3$\\

$N^{\g_2^{-\tfrac{2}{9}}}_7$ & $e_1e_1 = -\tfrac{2}{3} e_1 -\tfrac{11}{27}
e_2+e_3$, $e_1e_2 = -\tfrac{4}{3} e_2 +3 e_3$, 
$e_1e_3 = -\tfrac{4}{9}e_2+e_3$,\\
& $e_2e_1 = -\tfrac{4}{3} e_2 +2 e_3$, $e_2e_2 = -3e_2 +9e_3$, 
$e_2e_3 = -e_2+3e_3$, $e_3e_1 = -\tfrac{2}{9}e_2$, \\
& $e_3e_2 = -e_2 +3e_3$, $e_3e_3 = -\tfrac{1}{3} e_2+e_3$\\

\hline
\end{longtable}

Remarks:
\begin{itemize}
\item The algebras from Table \ref{tab3_2}, $N^{\g_2^\alpha}_1(a)$ for 
$\alpha = -\tfrac{2}{9}$ and $N^{\g_2^{a^2+a}}_2(a)$
for $a^2+a=-\tfrac{2}{9}$ are not in Table \ref{tab3_3}. The latter condition means
$a=-\tfrac{1}{3}$ or $a=-\tfrac{2}{3}$. This yields two non-isomorphic algebras.
\item Over $\C$ the algebra $N^{\g_2^{-\tfrac{2}{9}}}_7$ is isomorphic to 
$N^{\g_2^{-\tfrac{2}{9}}}_5(-\tfrac{2}{9})$. To describe the isomorphism, let $e_i$ be the basis
elements of $N^{\g_2^{-\tfrac{2}{9}}}_5(-\tfrac{2}{9})$ and $y_i$ those of $N^{\g_2^{-\tfrac{2}{9}}}_7$.
Then
\begin{align*}
e_1 &\mapsto y_1 +\frac{-2-\sqrt{-2}}{9}y_2\\
e_2 &\mapsto \frac{1-2\sqrt{-2}}{2}y_2 + \frac{-3+3\sqrt{-2}}{2}y_3\\
e_3 &\mapsto \frac{1-\sqrt{-2}}{3}y_2 + \frac{-2+\sqrt{-2}}{2} y_3
\end{align*}
defines an isomorphism $N^{\g_2^{-\tfrac{2}{9}}}_5(-\tfrac{2}{9}) \to N^{\g_2^{-\tfrac{2}{9}}}_7$.
Over $\R$ they are not isomorphic.
\end{itemize}

\begin{longtable}{|l|l|}
\caption{Novikov algebras with Lie algebra $\g_2^0$.}\label{tab3_4}
\endfirsthead
\hline
\multicolumn{2}{|l|}{\small\slshape Novikov algebras corresponding 
to $\g_2^0$.} \\ 
\hline
\endhead
\hline
\endfoot
\endlastfoot

\hline
name & multiplication table \\

\hline

$N^{\g_2^0}_8(a)$ & $e_1e_1 = a e_1$, $e_1e_2 = (a +1)e_3$,
$e_1e_3 = (a+1)e_3$, $e_2e_1 = a e_3$, $e_3e_1= a e_3$ \\

$N^{\g_2^0}_9$ & $e_1e_1 = -e_1+e_3$, 
$e_2e_1 = -e_3$, $e_3e_1= -e_3$ \\

$N^{\g_2^0}_{10}(a)$ & $e_1e_1 = a e_1$, $e_1e_2 = a e_2+e_3$,
$e_1e_3 = (a+1)e_3$, $e_2e_1=ae_2$, $e_2e_2=-e_2+e_3$, \\
& $e_3e_1= ae_3$ \\

$N^{\g_2^0}_{11}$ & $e_1e_1 = -e_1+e_3$, $e_1e_2 = -e_2+e_3$, $e_2e_1=-e_2$, $e_2e_2=-e_2+e_3$, 
$e_3e_1= -e_3$ \\

\hline
\end{longtable}

Remark: $N^{\g_2^0}_8(0)$ is $N^{\g_2^0}_1(0)$; so for the first algebra we
take $a\neq 0$.

\begin{longtable}{|l|l|}
\caption{Novikov algebras with Lie algebra $\g_3$.}\label{tab3_5}
\endfirsthead
\hline
\multicolumn{2}{|l|}{\small\slshape Novikov algebras corresponding 
to $\g_3$.} \\ 
\hline
\endhead
\hline
\endfoot
\endlastfoot

\hline
name & multiplication table \\

\hline

$N^{\g_3}_1(a)$ & $e_1e_1 = e_2$, $e_1e_2 = (a+1)e_3$, $e_2e_1=a e_3$\\ 

$N^{\g_3}_2(a)$ & $e_1e_1 = a e_3$, $e_1e_2 = e_3$, $e_2e_2= e_3$\\ 

$N^{\g_3}_3$ & $e_1e_1 = e_3$, $e_1e_2 = e_1$, $e_2e_1=e_1-e_3$,
$e_2e_2= e_2$, $e_2e_3=e_3$, $e_3e_2=e_3$\\ 

$N^{\g_3}_4$ & $e_1e_2 = e_1$, $e_2e_1=e_1-e_3$,
$e_2e_2= e_2$, $e_2e_3=e_3$, $e_3e_2=e_3$\\ 

$N^{\g_3}_5$ & $e_1e_2 = \tfrac{1}{2}e_3$, $e_2e_1=-\tfrac{1}{2}e_3$\\

\hline
\end{longtable}

\begin{longtable}{|l|l|}
\caption{Novikov algebras with Lie algebra $\g_4$.}\label{tab3_6}
\endfirsthead
\hline
\multicolumn{2}{|l|}{\small\slshape Novikov algebras corresponding 
to $\g_4$.} \\ 
\hline
\endhead
\hline
\endfoot
\endlastfoot

\hline
name & multiplication table \\

\hline

$N^{\g_4}_1(a)$ & $e_1e_1 = a e_1$, $e_1e_2 = ae_2+e_3$, $e_1e_3=e_2+ae_3$,
$e_2e_1=a e_2$, $e_3e_1=a e_3$\\ 

$N^{\g_4}_2$ & $e_1e_1 = e_1+e_3$, $e_1e_2 = e_2+e_3$, $e_1e_3=e_2+e_3$,
$e_2e_1= e_2$, $e_3e_1= e_3$\\ 

\hline
\end{longtable}

Remark: $N^{\g_4}_1(a)$ is isomorphic to $N^{\g_4}_1(b)$, if and only if
$a=b$ or $a=-b$. (In the latter case, $\phi(e_1)=-x_1$, $\phi(e_2)=x_2$,
$\phi(e_3) = -x_3$ defines an isomorphism. Here the $x_i$ are the basis elements
of $N^{\g_4}_1(b)$.)

\begin{longtable}{|l|l|}
\caption{Novikov algebras over $\R$ with Lie algebra $\g_5$.}\label{tab3_7}
\endfirsthead
\hline
\multicolumn{2}{|l|}{\small\slshape Novikov algebras corresponding 
to $\g_5$.} \\ 
\hline
\endhead
\hline
\endfoot
\endlastfoot

\hline
name & multiplication table \\

\hline

$N^{\g_5}_1(a)$ & $e_1e_1 = a e_1$, $e_1e_2 = ae_2+e_3$, $e_1e_3=-e_2+ae_3$,
$e_2e_1=a e_2$, $e_3e_1=a e_3$\\ 

\hline
\end{longtable}

Remarks:
\begin{itemize}
\item Let $e_i$ denote the basis elements of $N^{\g_5}_1(a)$, and let the base field be $\C$.
Then setting $x_1 = ie_1$, $x_2=e_2$, $x_3= ie_3$, we see that the $x_i$ satisfy the multiplication
table of $N^{\g_4}_1(ia)$. Hence, over $\C$ we have that $N^{\g_5}_1(a)$ is isomorphic to 
$N^{\g_4}_1(ia)$. Over $\R$ such an isomorphism does not exist as the 
underlying Lie algebras are not isomorphic.
\item Table \ref{tab3_7} gives a list of all Novikov algebras over $\R$, with Lie 
algebra $\g_5$. Over $\C$ there would be an extra algebra isomorphic to $N_2^{\g_4}$; however
over $\R$ this algebra does not exist.
\item Also here $N^{\g_5}_1(a)$ is isomorphic to $N^{\g_5}_1(b)$ if and
only if $a=b$ or $a=-b$.
\end{itemize}

\section{Novikov algebras over $\C$ of dimension four with a nilpotent Lie algebra}
\label{sec:dim4}

In this section we give the classification of the 4-dimensional Novikov algebras
over $\C$ such that the associated Lie algebra is nilpotent.

First of all there are the Novikov algebras that have an abelian associated Lie
algebra. Again those are commutative associative algebras. They can be classified
using the same procedure as in Section \ref{sec:caa}, using the classification of
nilpotent CAA's up to dimension 4. This can be obtained from \cite{poonen}, or
\cite{gra_ass}. For the nilpotent CAA's of dimensions up to 3 we use the multiplication tables
of Table \ref{tab:nilcaas}. The multiplication tables of the nilpotent CAA's of dimension 4
are taken from \cite{gra_ass}. We get the following list of CAA's of dimension 4.

\newpage

\begin{longtable}{|r|l|}
\caption{CAA's of dimension 4.}\label{tab2}
\endfirsthead
\hline
\multicolumn{2}{|l|}{\small\slshape CAA's of dimension 4.} \\ 
\hline
\endhead
\hline
\endfoot
\endlastfoot

\hline
name & multiplication table \\

\hline

$A_{4,1}$ & \\
$A_{4,2}$ & $e_1^2=e_2$\\
$A_{4,3}$ & $e_1^2=e_3$, $e_2^2=e_3$\\
$A_{4,4}$ & $e_1^2 = e_2$, $e_1e_2=e_3$\\
$A_{4,5}$ & $e_1^2= -e_3$, $e_1e_2=e_4$, $e_2^2=e_3$\\
$A_{4,6}$ & $e_1e_2=e_4$, $e_2^2=e_3$\\
$A_{4,7}$ & $e_1^2=e_4$, $e_2^2=e_4$, $e_3^2=e_4$\\
$A_{4,8}$ & $e_1^2=e_2$, $e_1e_2=e_4$, $e_3^2=e_4$\\
$A_{4,9}$ & $e_1^2=e_2$, $e_1e_2=e_3$, $e_1e_3=e_4$, $e_2^2=e_4$\\

$4\wt{A}_0$ & $e_1^2=e_1$, $e_2^2=e_2$, $e_3^2=e_3$, $e_4^2=e_4$\\

$2\wt{A}_0+\wt{A}_1$ & $e_1^2=e_1$, $e_2^2=e_2$, $e_3^2=e_3$, $e_3e_4=e_4$\\

$2\wt{A}_1$ & $e_1^2=e_1$, $e_1e_2=e_2$, $e_3^2=e_3$, $e_3e_4=e_4$\\

$\wt{A}_0+\wt{A}_{2,1}$ & $e_1^2=e_1$, $e_2^2=e_2$, $e_2e_3=e_3$, $e_2e_4=e_4$\\

$\wt{A}_0+\wt{A}_{2,2}$ & $e_1^2=e_1$, $e_2^2=e_2$, $e_2e_3=e_3$, $e_2e_4=e_4$,
$e_3^2=e_4$\\

$\wt{A}_{3,1}$ & $e_1^2=e_1$, $e_1e_2=e_2$, $e_1e_3=e_3$, $e_1e_4=e_4$\\

$\wt{A}_{3,2}$ & $e_1^2=e_1$, $e_1e_2=e_2$, $e_1e_3=e_3$, $e_1e_4=e_4$, 
$e_2^2=e_3$\\

$\wt{A}_{3,3}$ & $e_1^2=e_1$, $e_1e_2=e_2$, $e_1e_3=e_3$, $e_1e_4=e_4$,
$e_2^2=e_3$, $e_2e_3=e_4$\\

$\wt{A}_{3,4}$ & $e_1^2=e_1$, $e_1e_2=e_2$, $e_1e_3=e_3$, $e_1e_4=e_4$
$e_2^2=e_4$, $e_3^2=e_4$ \\

$3\wt{A}_0+A_1$ & $e_1^2=e_1$, $e_2^2=e_2$, $e_3^2=e_3$ \\
$\wt{A}_0+\wt{A}_1+A_1$ & $e_1^2=e_1$, $e_2^2=e_2$, $e_2e_3=e_3$ \\
$\wt{A}_{2,1}+A_1$ & $e_1^2=e_1$, $e_1e_2=e_2$, $e_1e_3=e_3$ \\
$\wt{A}_{2,2}+A_1$ & $e_1^2=e_1$, $e_1e_2=e_2$, $e_1e_3=e_3$, $e_2^2=e_3$ \\

$2\wt{A}_0+A_{2,1}$ &  $e_1^2=e_1$, $e_2^2=e_2$ \\
$2\wt{A}_0+A_{2,2}$ &  $e_1^2=e_1$, $e_2^2=e_2$, $e_3^2=e_4$ \\
$\wt{A}_1+A_{2,1}$ & $e_1^2=e_1$, $e_1e_2=e_2$\\
$\wt{A}_1+A_{2,2}$ & $e_1^2=e_1$, $e_1e_2=e_2$, $e_3^2=e_4$\\

$\wt{A}_0+A_{3,1}$ & $e_1^2=e_1$ \\
$\wt{A}_0+A_{3,2}$ & $e_1^2=e_1$, $e_2e_3=e_4$ \\
$\wt{A}_0+A_{3,3}$ & $e_1^2=e_1$, $e_2^2=e_3$ \\
$\wt{A}_0+A_{3,4}$ & $e_1^2=e_1$, $e_2^2=e_3$, $e_2e_3=e_4$ \\

\hline

\end{longtable}

There are the following non-abelian nilpotent Lie algebras of dimension 4:

\begin{center}
\begin{tabular}{|l|l|}
\hline 
name & nonzero brackets \\
\hline
$\hh_1$ & $[e_1,e_2]=e_3$\\ 
$\hh_2$ & $[e_1,e_2] = e_3$, $[e_1,e_3]=e_4$ \\
\hline
\end{tabular}
\end{center}

Tabel \ref{tab3} contains the Novikov algebras with associated Lie algebra $\hh_1$.
Table \ref{tab4} has those with Lie algebra $\hh_2$.

\newpage

\begin{longtable}{|l|l|}
\caption{Novikov algebras with Lie algebra $\hh_1$.}\label{tab3}
\endfirsthead
\hline
\multicolumn{2}{|l|}{\small\slshape Novikov algebras corresponding to $\hh_1$.} \\ 
\hline
\endhead
\hline
\endfoot
\endlastfoot

\hline
name & multiplication table \\

\hline

$N^{\hh_1}_1(\alpha)$ & $e_1e_2 = (\alpha+1) e_3, ~ e_2e_1 = \alpha e_3$ \\
$N^{\hh_1}_2(\alpha)$ & $e_1e_2 = (\alpha+1) e_3, ~ e_2e_1 = \alpha e_3, 
e_2e_2=e_4$\\ 
$N^{\hh_1}_3(\alpha)$ &
$e_1e_2 = (\alpha+1) e_3, ~ e_2e_1 = \alpha e_3, e_2e_2=e_1$\\
$N^{\hh_1}_4$ & $e_1e_1 = e_3, ~ e_1e_2 = e_3, e_2e_2=e_4$\\
$N^{\hh_1}_5$ & $e_1e_2=e_3,~ e_2e_2=e_4,~ e_2e_4=e_3,~ e_4e_2=e_3$\\
$N^{\hh_1}_6$ & $e_1e_2=e_3,~ e_2e_4=e_3,~ e_4e_2=e_3$\\
$N^{\hh_1}_7(\alpha)$ & $e_1e_2=e_3,~ e_2e_2=e_1+\alpha e_4,~ e_2e_4=e_3,~ 
e_4e_2=e_3$\\
$N^{\hh_1}_8$ &
$e_1e_1=e_3,~ e_1e_2=e_3,~ e_2e_2= e_4,~ e_2e_4=e_3,~ e_4e_2=e_3$\\
$N^{\hh_1}_9$ & $e_1e_1=e_3,~ e_1e_2=e_3,~ e_2e_4=e_3,~ e_4e_2=e_3$\\
$N^{\hh_1}_{10}(\alpha)$ &
$e_1e_2 = (\alpha+1) e_3, ~ e_2e_1 = \alpha e_3,~ e_4e_4=e_3$\\
$N^{\hh_1}_{11}$ & $e_1e_2 = \tfrac{1}{2}e_3,~ e_2e_1 = -\tfrac{1}{2}e_3,~ 
e_2e_2=e_3,~ e_4e_4=e_3$\\
$N^{\hh_1}_{12}(\alpha)$ &
$e_1e_2 = (\alpha+1) e_3, ~ e_2e_1 = \alpha e_3, e_2e_2=e_1,~
e_4e_4 = e_3$\\
$N^{\hh_1}_{13}$ & $e_1e_2=e_3+e_4,~ e_2e_1=e_4$\\
$N^{\hh_1}_{14}$ & $e_1e_2=e_3+e_4,~ e_2e_1=e_4,~ e_2e_2=e_1$\\
$N^{\hh_1}_{15}(\alpha)$ &
$e_1e_1=e_3,~ e_1e_2=e_3+e_4,~ e_2e_1=e_4,~ e_2e_2=\alpha e_3$\\
$N^{\hh_1}_{16}$ &
$e_1e_1=e_3,~ e_1e_2=e_3+e_4,~ e_2e_1=e_4,~ e_2e_2=e_1,~ e_2e_4=e_3,~
e_4e_2=e_3$\\
$N^{\hh_1}_{17}$ &
$e_1e_1=e_1,~ e_1e_2=e_2+e_3,~ e_1e_3=e_3,~ e_2e_1=e_2,~
e_2e_2=e_3,~ e_3e_1=e_3$\\
$N^{\hh_1}_{18}$ &
$e_1e_1=e_1,~ e_1e_2=e_2+e_3,~ e_1e_3=e_3,~ e_2e_1=e_2,~
e_3e_1=e_3$\\
$N^{\hh_1}_{19}$ & 
$e_1e_1=e_1,~ e_1e_2=e_2+e_3,~ e_1e_3=e_3,~ e_1e_4=e_4,~ e_2e_1=e_2,~
e_3e_1=e_3,~ e_4e_1=e_4$\\
$N^{\hh_1}_{20}$ & $e_1e_1=e_1,~ e_1e_2=e_2+e_3,~ e_1e_3=e_3,~ e_1e_4=e_4,~ 
e_2e_1=e_2,~e_2e_2=e_2,~ e_3e_1=e_3,$ \\
& $e_4e_1=e_4$\\
$N^{\hh_1}_{21}$ &
$e_1e_1=e_1,~ e_1e_2=e_2+e_3,~ e_1e_3=e_3,~ e_1e_4=e_4,~ e_2e_1=e_2,~
e_2e_2=e_4,~ e_3e_1=e_3$,\\ 
& $e_4e_1=e_4$\\
$N^{\hh_1}_{22}$ & 
$e_1e_1=e_1,~ e_1e_2=e_2+e_3,~ e_1e_3=e_3,~ e_1e_4=e_4,~ e_2e_1=e_2,~
e_2e_4=e_3,~ e_3e_1=e_3$,\\ 
& $e_4e_1=e_4,$ $e_4e_2=e_3$\\
$N^{\hh_1}_{23}$ & $e_1e_1=e_1,~ e_1e_2=e_2+e_3,~ e_1e_3=e_3,~ e_1e_4=e_4,~ 
e_2e_1=e_2,~
e_2e_2=e_4,~ e_2e_4=e_3$,\\
& $e_3e_1=e_3,$ $e_4e_1=e_4,~ e_4e_2=e_3$\\
$N^{\hh_1}_{24}$ & $e_1e_1=e_1,~ e_1e_2=e_2+e_3,~ e_1e_3=e_3,~ e_1e_4=e_4,~ 
e_2e_1=e_2,~e_2e_2=e_3,~ e_3e_1=e_3$,\\
& $ e_4e_1=e_4,$ $ e_4e_4=e_3$\\
$N^{\hh_1}_{25}$ &
$e_1e_1=e_1,~ e_1e_2=e_2+e_3,~ e_1e_3=e_3,~ e_1e_4=e_4,~ e_2e_1=e_2,~
e_3e_1=e_3,~ e_4e_1=e_4$,\\
& $ e_4e_4=e_3$\\
$N^{\hh_1}_{26}(\alpha)$ & 
$e_1e_2 = (\alpha+1)e_3,~ e_2e_1 = \alpha e_3, ~ e_2e_2=e_1,~e_4e_4=e_4$\\
$N^{\hh_1}_{27}(\alpha)$ &
$e_1e_2 = (\alpha+1)e_3,~ e_2e_1 = \alpha e_3, ~ e_4e_4=e_4$\\
$N^{\hh_1}_{28}$ &
$e_1e_2 = \tfrac{1}{2}e_3,~ e_2e_1 = -\tfrac{1}{2}e_3,~ e_2e_2=e_3,~
e_4e_4=e_4$\\
$N^{\hh_1}_{29}$ &
$e_1e_1=e_1,~ e_1e_2=e_2+e_3,~ e_1e_3=e_3, ~ e_2e_1=e_2,~
e_2e_2=e_3,~ e_3e_1 = e_3,~ e_4e_4=e_4$\\
$N^{\hh_1}_{30}$ &
$e_1e_1=e_1,~ e_1e_2=e_2+e_3,~ e_1e_3=e_3, ~ e_2e_1=e_2,~
e_3e_1 = e_3,~ e_4e_4=e_4$\\

\hline
\end{longtable}

Among these algebras there are precisely the following isomorphisms:

\begin{itemize}
\item $N_1^{\hh_1}(\alpha) \cong N_1^{\hh_1}(\beta)$ if and only if 
$\alpha = \beta$ or
$\beta = -\alpha-1$. In the latter case, $\phi(e_1) = -y_2$, $\phi(e_2) = 
y_1$, $\phi(e_i)=y_i$, $i=3,4$ defines an isomorphism 
$\phi : N_1^{\hh_1}(\alpha) \to
N_1^{\hh_1}(-\alpha-1)$.

\item $N_{10}^{\hh_1}(\alpha) \cong N_{10}^{\hh_1}(\beta)$ if and only if 
$\alpha = \beta$ or
$\beta = -\alpha-1$. In the latter case, $\phi(e_1) = -y_2$, $\phi(e_2) = 
y_1$, $\phi(e_i)=y_i$, $i=3,4$ defines an isomorphism 
$\phi : N_{10}^{\hh_1}(\alpha)
\to N_{10}^{\hh_1}(-\alpha-1)$. 
\end{itemize}

\begin{longtable}{|l|l|}
\caption{Novikov algebras with Lie algebra $\hh_2$.}\label{tab4}
\endfirsthead
\hline
\multicolumn{2}{|l|}{\small\slshape Novikov algebras corresponding to $\hh_2$.} \\ 
\hline
\endhead
\hline
\endfoot
\endlastfoot

\hline
name & multiplication table \\

\hline

$N_1^{\hh_2}$ & $e_1e_2=e_3,~e_1e_3=e_4$\\
$N_2^{\hh_2}$ & $e_1e_1=e_2,~e_1e_2=e_3,~e_1e_3=e_4$\\
$N_3^{\hh_2}$ & $e_1e_2=e_3+e_4,~e_1e_3=e_4,~e_2e_1=e_4$\\
$N_4^{\hh_2}$ & $e_1e_1=e_2,~e_1e_2=e_3+e_4,~e_1e_3=e_4,~e_2e_1=e_4$\\
$N_5^{\hh_2}$ & $e_1e_3=\tfrac{1}{2}e_4,~e_2e_1=-e_3,~e_3e_1=-\tfrac{1}{2}e_4$\\
$N_6^{\hh_2}$ & $e_1e_1=e_4,~e_1e_3=\tfrac{1}{2}e_4,~e_2e_1=-e_3,~
e_3e_1=-\tfrac{1}{2}e_4$\\
$N_7^{\hh_2}$ & $e_1e_1=e_2,~e_1e_3=\tfrac{1}{2}e_4,~e_2e_1=-e_3,~
e_3e_1=-\tfrac{1}{2}e_4$\\
$N_8^{\hh_2}(\alpha)$ & $e_1e_1=(2\alpha^2+\alpha)e_2,~ e_1e_2=(2\alpha +1)e_3,~ 
e_1e_3=(\alpha+1) e_4,~ e_2e_1 = 2\alpha e_3,~e_2e_2=e_4$,\\
& $e_3e_1= \alpha e_4$\\ 
$N_9^{\hh_2}$ & $e_1e_1=e_3,~ e_1e_2=e_3,~ e_1e_3=e_4,~ e_2e_2=e_4$\\
$N_{11}^{\hh_2}$ & $e_1e_1=e_3,~e_1e_3=\tfrac{1}{2}e_4,~e_2e_1=-e_3,~e_2e_2=e_4,~
e_3e_1=-\tfrac{1}{2} e_4$\\
$N_{12}^{\hh_2}$ & $e_1e_1=e_4,~ e_1e_2=e_3,~e_1e_3=e_4,~e_2e_2=2 e_3,~ 
e_2e_3=e_4,~ e_3e_2=e_4$\\
$N_{13}^{\hh_2}$ & $e_1e_2=e_3,~e_1e_3=e_4,~e_2e_2=2 e_3,~ e_2e_3=e_4,~
e_3e_2=e_4$\\
$N_{14}^{\hh_2}(\alpha)$ & $e_1e_1=\alpha e_4,~ e_1e_2=e_3,~e_1e_3=e_4,~
e_2e_2=2 e_3+e_4,~ e_2e_3=e_4,~e_3e_2=e_4$\\ 
$N_{15}^{\hh_2}$ & $e_1e_1=e_1,~ e_1e_2=e_2+e_3,~e_1e_3=e_3+e_4,~e_1e_4=e_4,~
e_2e_1=e_2,~ e_3e_1=e_3$,\\
& $e_4e_1=e_4$\\
$N_{16}^{\hh_2}$ & $e_1e_1=e_1,~ e_1e_2=e_2+e_3,~e_1e_3=e_3+e_4,~e_1e_4=e_4,~
e_2e_1=e_2,~ e_2e_2=e_4$,\\
& $e_3e_1=e_3,~e_4e_1=e_4$\\
$N_{17}^{\hh_2}$ & $e_1e_1=e_1,~ e_1e_2=e_2+e_3,~e_1e_3=e_3+e_4,~e_1e_4=e_4,~
e_2e_1=e_2,~ e_2e_2=2e_3+\alpha e_4$,\\
& $e_2e_3=e_4,~e_3e_1=e_3,~e_3e_2=e_4,~
e_4e_1=e_4$ 
\\
\hline
\end{longtable}

\end{document}